\documentclass[envcountsame,runningheads,a4paper]{llncs}
\pdfoutput=1

  \pdfpageattr{/Group <</S /Transparency /I true /CS /DeviceRGB>>}

\usepackage[T1]{fontenc}
\usepackage[english]{babel}
\usepackage{hyperref}
\usepackage{geometry}
\usepackage{pdflscape}
\usepackage{amsfonts,amsmath,amssymb}
\usepackage{tikz}
\usepackage{xspace}
\usepackage{booktabs}
\usepackage{paralist}
\usepackage{dcolumn}
\usepackage{listings}
\usepackage{marvosym}
\usepackage{microtype}

\usepackage[vlined, ruled, oldcommands]{algorithm2e}

\microtypesetup{tracking,kerning,spacing}
\microtypecontext{spacing=nonfrench}

\usetikzlibrary{calc, 3d}

\usetikzlibrary{decorations.markings}

\newcommand{\Gitter}[4]{
    \draw[very thin,color=gray] (#1,#3) grid (#2,#4);
}

\newcommand{\Koordinatenkreuz}[6]{
    \draw[->, >=latex, color=black!50, thick] (#1,0) -- (#2,0) node[right] {#5};
    \draw[->, >=latex, color=black!50, thick] (0,#3) -- (0,#4) node[left] {#6};
}

\definecolor{links}{rgb}{.2,.1,.5}
\definecolor{cites}{rgb}{.5,.1,.2}
\hypersetup{colorlinks,linkcolor=links,citecolor=cites,urlcolor=links}

\newcolumntype{d}[1]{D{.}{.}{#1} }

\newcommand\polymake{{\tt polymake}\xspace}

\newcommand\perl{{\tt Perl}\xspace}
\newcommand\cplusplus{{\tt C++}\xspace}

\newcommand\FF{{\mathbb F}}
\newcommand\NN{{\mathbb N}}

\newcommand\QQ{{\mathbb Q}}
\newcommand\RR{{\mathbb R}}

\newcommand\TT{{\mathbb T}}

\newcommand\SetOf[2]{\left\{\left.#1\vphantom{#2}\ \right|\ #2\vphantom{#1}\right\}}
\newcommand\smallSetOf[2]{\{{#1}\,|\,{#2}\}}

\DeclareMathOperator{\val}{val}

\DeclareMathOperator{\tcone}{tcone}
\DeclareMathOperator{\tdet}{tdet}
\DeclareMathOperator{\sgn}{sgn}
\DeclareMathOperator{\LP}{LP}
\newcommand\puiseux[2]{#1\{#2\}} 
\newcommand\pseries[2]{#1\{\!\!\{#2\}\!\!\}} 

\newcommand\transpose[1]{{#1}^\top}

\usepackage[colorinlistoftodos,bordercolor=orange,backgroundcolor=orange!20,linecolor=orange,textsize=scriptsize]{todonotes}

\title{Linear programs and convex hulls\\ over fields of Puiseux fractions}
\titlerunning{Linear programs and convex hulls over fields of Puiseux fractions}

\author{Michael Joswig\thanks{Partially supported by Einstein Foundation Berlin and Deutsche Forschungsgemeinschaft (DFG) within the Priority Program 1489 ``Experimental Methods in Algebra, Geometry and Number Theory''.} 
\and Georg Loho \and Benjamin Lorenz \and Benjamin Schr\"oter}
\authorrunning{Joswig, Loho, Lorenz, Schr\"oter}

\institute{
Institut f{\"u}r Mathematik,
 TU Berlin, MA 6-2,\\
 Str.\ des 17. Juni 136, 10623 Berlin, Germany\\
\email{\href{mailto:<lastname>@math.tu-berlin.de}{<joswig, loho, lorenz, schroeter>@math.tu-berlin.de}}
}

\begin{document}
\maketitle
\begin{abstract}
  We describe the implementation of a subfield of the field of formal Puiseux series in \polymake.  This is employed for
  solving linear programs and computing convex hulls depending on a real parameter.  Moreover, this approach is also useful for
  computations in tropical geometry.
\end{abstract}

\section{Introduction}
\noindent
It is well known and not difficult to see that the standard concepts from linear programming (LP), e.g., the Farkas
Lemma and LP duality, carry over to an arbitrary ordered field; e.g., see \cite[Section~II]{CharnesKortanek:1970} or
\cite[\S2.1]{Jeroslow:1973b}.  Traces of this can already be found in Dantzig's monograph~\cite[Chapter
22]{Dantzig:1963}.  This entails that any algorithm whose correctness rests on these LP corner stones is valid over any
ordered field.  In particular, this holds for the simplex method and usual convex hull algorithms.  A classical
construction, due to Hilbert, turns a field of rational functions, e.g., with real coefficients, into an ordered field;
see \cite[\S 147]{vanderWaerden:Algebra2}.  In \cite{Jeroslow:1973b} Jeroslow discussed these fields in the context of
linear programming in order to provide a rigorous foundation of the so-called ``big M method''.  The purpose of this
note is to describe the implementation of the simplex method and of a convex hull algorithm over fields of this kind in
the open source software system \polymake~\cite{DMV:polymake}.

Hilbert's ordered field of rational functions is a subfield of the field of formal Puiseux series $\pseries{\RR}{t}$
with real coefficients.  The latter field is real-closed by the Artin--Schreier
Theorem \cite[Theorem~12.10]{SGHL:2007}; by Tarski's Principle (cf. \cite{Tarski:1948}) this implies that
$\pseries{\RR}{t}$ has the same first order properties as the reals.  The study of polyhedra over $\pseries{\RR}{t}$ is
motivated by tropical geometry \cite{DevelinYu:2007}, especially tropical linear programming~\cite{AlBeGaJo:2015}.  The
connection of the latter with classical linear programming has recently lead to a counter-example \cite{ABGJ:1405.4161}
to a ``continuous analogue of the Hirsch conjecture'' by Deza, Terlaky and Zinchenko \cite{DezaTerlakyZinchenko:2009}.
In terms of parameterized linear optimization (and similarly for the convex hull computations) our approach amounts to
computing with sufficiently large (or, dually, sufficiently small) positive real numbers.  Here we do \emph{not} consider the
more general algorithmic problem of stratifying the parameter space to describe all optimal solutions of a linear
program for \emph{all} choices of parameters; see, e.g., \cite{JCM:2008} for work into that direction.

This paper is organized as follows. We start out with summarizing known facts on ordered fields.  Then we describe a
specific field, $\puiseux{\QQ}{t}$, which is the field of rational functions with rational coefficients and rational
exponents.  This is a subfield of $\pseries{\QQ}{t}$, which we call the field of \emph{Puiseux fractions}.  It is our
opinion that this is a subfield of the formal Puiseux series which is particularly well suited for exact computations
with (some) Puiseux series; see \cite{ManaaCoquand:2013} for an entirely different approach.  In the context of tropical
geometry Markwig~\cite{Markwig:2010} constructed a much larger field, which contains the classical Puiseux series as a proper
subfield.  For our applications it is relevant to study the evaluation of Puiseux fractions at sufficiently large
rational numbers.  In Section~\ref{sec:polyhedra} we develop what this yields for comparing convex polyhedra over
$\pseries{\RR}{t}$ with ordinary convex polyhedra over the reals.  The tropical geometry point of view enters the
picture in Section~\ref{sec:tropical}.  We give an algorithm for solving the dual tropical convex hull problem, i.e.,
the computation of generators of a tropical cone from an exterior description.  Allamigeon, Gaubert and Goubault gave a
combinatorial algorithm for this in~\cite{AllamigeonGaubertGoubault:2013}, while we use a classical (dual) convex hull
algorithm and apply the valuation map.  The benefit of our approach is more geometric than in terms of computational
complexity: in this way we will be able to study the fibers of the tropicalization map for classical versus tropical
cones for specific examples.  Section~\ref{sec:implementation} sketches the \polymake implementation of the Puiseux
fraction arithmetic and the LP and convex hull algorithms.  The LP solver is a dual simplex algorithm with steepest edge
pivoting, and the convex hull algorithm is the classical beneath-and-beyond method \cite{Edelsbrunner:1987} \cite{Joswig:2003}.  An overview with computational
results is given in Section~\ref{sec:computations}.

\subsubsection*{Acknowledgment}

We thank Thomas Opfer for contributing to and maintaining within the \polymake project his implementation of the dual
simplex method, originally written for his Master's Thesis~\cite{ThomasDipl}.

\section{Ordered fields and rational functions}\label{sec:puiseux}
\noindent
A field $\FF$ is \emph{ordered} if there is a total ordering $\leq$ on the set $\FF$ such that for all $a,b,c\in\FF$ the
following conditions hold:
\begin{compactenum}[(i)]
\item if $a \leq b$ then $a + c \leq b + c$,
\item if $0 \leq a$ and $0\leq b$ then $0 \leq a \cdot b$.
\end{compactenum}
Any ordered field necessarily has characteristic zero.  Examples include the rational numbers $\QQ$, the reals $\RR$ and
any subfield in between.

Given an ordered field $\FF$ we can look at the ring of univariate polynomials $\FF[t]$ and its quotient field $\FF(t)$,
the field of rational functions in the indeterminate $t$ with coefficients in~$\FF$.  On the ring $\FF[t]$ we obtain
a total ordering by declaring $p < q$ whenever the leading coefficient of $q-p$ is a positive element in $\FF$.
Extending this ordering to the quotient field by letting
\[
\frac{u}{v} < \frac{p}{q} \ :\iff \ u q < v p \enspace ,
\]
where the denominators $v$ and $q$ are assumed positive, turns $\FF(t)$ into an ordered field; see, e.g.,
\cite[\S147]{vanderWaerden:Algebra2}.  This ordered field is called the ``Hilbert field'' by Jeroslow \cite{Jeroslow:1973b}.

By definition, the exponents of the polynomials in $\FF[t]$ are natural numbers.  However, conceptually, there is no
harm in also taking negative integers or even arbitrary rational numbers as exponents into account, as this can be
reduced to the former by clearing denominators and subsequent substitution.  For example,
\begin{equation}\label{eq:example_fraction}
  \frac{2t^{3/2}-t^{-1}}{1+3t^{-1/3}} \ = \ \frac{2t^{5/2}-1}{t+3t^{2/3}} \ = \ \frac{2s^{15}-1}{s^6+3s^4} \enspace ,
\end{equation}
where $s=t^{1/6}$.  In this way that fraction is written as an element in the field $\QQ(t^{1/6})$ of rational functions
in the indeterminate $s=t^{1/6}$ with rational coefficients.  Further, if $p\in\FF(t^{1/\alpha})$ and
$q\in\FF(t^{1/\beta})$, for natural numbers $\alpha$ and $\beta$, then the sum $p+q$ and the product $p \cdot q$ are
contained in $\FF(t^{1/\gcd(\alpha,\beta)})$.  This shows that the union
\begin{equation}
  \puiseux{\FF}{t} \ = \ \bigcup_{\nu \geq 1} \FF(t^{1/\nu})
\end{equation}
is again an ordered field.  We call its elements \emph{Puiseux fractions}.  The field $\puiseux{\FF}{t}$ is a subfield
of the field $\pseries{\FF}{t}$ of \emph{formal Puiseux series}, i.e., the formal power series with rational exponents
of common denominator.  For an algorithmic approach to general Puiseux series see \cite{ManaaCoquand:2013}.

The map $\val$ which sends the rational function $p/q$, where $p,q\in\FF[t^{1/\nu}]$, to the number $\deg_tp-\deg_tq$
defines a non-Archimedean valuation on $\FF(t)$.  Here we let $\val(0)=\infty$.  As usual the \emph{degree} is the
largest occurring exponent.  The valuation map extends to Puiseux series.  More precisely, for $f,g\in \puiseux{\FF}{t}$
we have the following:
\begin{compactenum}[(i)]
\item $\val(f\cdot g) = \val(f)+\val(g)$,
\item $\val(f+g) \leq \max(\val(f),\val(g))$.
\end{compactenum}

If $\FF=\RR$ is the field of real numbers we can evaluate a Puiseux fraction $f\in\puiseux{\RR}{t}$ at a real number
$\tau$ to obtain the real number $f(\tau)$.  This map is defined for all $\tau>0$ except for the finitely many poles,
i.e., zeros of the denominator.  Restricting the evaluation to positive numbers is necessary since we are allowing
rational exponents.  The valuation map satisfies the equation
\begin{equation}\label{eq:limit}
  \lim_{\tau\to\infty} \log_{\tau} |f(\tau)| \ = \ \val(f) \enspace .
\end{equation}
That is, seen on a logarithmic scale, taking the valuation of $f$ corresponds to interpreting $t$ like an
infinitesimally large number.  Reading the valuation map in terms of the limit \eqref{eq:limit} is known as
\emph{Maslov dequantization}, see \cite{Maslov:1986}.

Occasionally, it is also useful to be able to interpret $t$ as a \emph{small} infinitesimal.  To this end, one can
define the \emph{dual degree} $\deg^*$, which is the smallest occurring exponent.  This gives rise to the \emph{dual valuation}
map $\val^*(p/q)=\deg^*_tp-\deg^*_tq$ which yields
\[
\val^*(f+g) \ \geq \ \min(\val^*(f),\val^*(g)) \quad \text{and} \quad \lim_{\tau\to 0} \log_{\tau} |f(\tau)| \ = \ \val^*(f) \enspace .
\]
Changing from the primal to the dual valuation is tantamount to substituting $t$ by $t^{-1}$.

\begin{remark}
  The valuation theory literature often employs the dual definition of a valuation.  The equation~\eqref{eq:limit} is
  the reason why we usually prefer to work with the primal.
\end{remark}

Up to isomorphism of valuated fields the valuation on the field $\FF(t)$ of rational functions is unique, e.g., see
\cite[\S 147]{vanderWaerden:Algebra2}.  As a consequence the valuation on the slightly larger field of Puiseux fractions
is unique, too.

To close this section let us look at the algorithmically most relevant case $\FF=\QQ$.  Then, in general, the evaluation
map sends positive rationals to not necessarily rational numbers, again due to fractional exponents.  By clearing
denominators in the exponents one can see that evaluating at $\sigma>0$ ends up in the totally real number field
$\QQ(\sqrt[\nu]{\sigma})$ for some positive integer $\nu$.  For instance, evaluating the Puiseux fraction from
Example~\eqref{eq:example_fraction} would give an element of $\QQ(\sqrt[6]{\sigma})$.

\goodbreak

\section{Parameterized Polyhedra}\label{sec:polyhedra}
\noindent
Consider a matrix $A \in \puiseux{\FF}{t}^{m \times (d+1)}$.  Then the set
\[
C \ := \ \SetOf{x\in\puiseux{\FF}{t}^{d+1}}{A\cdot x \geq 0}
\]
is a polyhedral cone in the vector space $\puiseux{\FF}{t}^{d+1}$.  Equivalently, $C$ is the set of feasible solutions
of a linear program with $d+1$ variables over the ordered field $\puiseux{\FF}{t}$ with $m$ homogeneous constraints, the rows of $A$.
The Farkas--Minkowski--Weyl Theorem establishes that each polyhedral cone is finitely generated.  A proof for this
result on polyhedral cones over the reals can be found in \cite[\S1.3 and \S1.4]{Ziegler:1995} under the name ''Main theorem for cones''.
It is immediate to verify that the arguments given hold over any ordered field.  Therefore, there is a matrix $B\in
\puiseux{\FF}{t}^{(d+1)\times n}$, for some $n\in\NN$, such that
\begin{equation}\label{eq:B}
  C \ = \ \SetOf{B\cdot a}{a\in\puiseux{\FF}{t}^n,\, a\geq 0} \enspace .
\end{equation}
The columns of $B$ are points and the cone $C$ is the non-negative linear span of those.

Let $L$ be the \emph{lineality space} of $C$, i.e., $L$ is the unique maximal linear subspace of
$\puiseux{\FF}{t}^{d+1}$ which is contained in $C$.  If $\dim L=0$ the cone $C$ is \emph{pointed}.  Otherwise, the set
$C/L$ is a pointed polyhedral cone in the quotient space $\puiseux{\FF}{t}^{d+1}/L$.  A \emph{face} of $C$ is the
intersection of $C$ with a supporting hyperplane.  The faces are partially ordered by inclusion.
Each face contains the
lineality space.  Adding the entire cone $C$ as an additional top element we obtain a lattice, the \emph{face lattice}
of $C$.  The maximal proper faces are the \emph{facets} which form the co-atoms in the face lattice. 
The \emph{combinatorial type} of $C$ is the isomorphism class of the face lattice (e.g., as a partially ordered
set).  Notice that our definition says that each cone is combinatorially equivalent to its quotient modulo its lineality
space.

Picking a positive element $\tau$ yields matrices $A(\tau)\in\FF^{m\times(d+1)}$ and $B(\tau)\in\FF^{(d+1)\times n}$ as
well as a polyhedral cone $C(\tau) = \smallSetOf{x\in\FF^{d+1}}{A(\tau)\cdot x \geq 0}$ by evaluating the Puiseux
fractions at the parameter $\tau$.  Here and below we will assume that $\tau$ avoids the at most finitely many poles of
the $(m+n)\cdot(d+1)$ coefficients of $A$ and $B$.

\begin{theorem}\label{thm:evaluation}
  There is a positive element $\tau_0\in\FF$ so that for every $\tau > \tau_0$ we have
  \[
  C(\tau) \ = \ \SetOf{B(\tau)\cdot\alpha}{\alpha\in\FF^n,\,\alpha\geq 0} \enspace ,
  \]
  and evaluating at $\tau$ maps the lineality space of $C$ to the lineality space of $C(\tau)$.  Moreover, the
  polyhedral cones $C$ and $C(\tau)$ over $\puiseux{\FF}{t}$ and $\FF$, respectively, share the same combinatorial
  type. 
\end{theorem}
\begin{proof}
  First we show that an orthogonal basis of the lineality space $L$ evaluates to an orthogonal basis of the lineality space of $C(\tau)$.
  For this, consider two vectors $x,y \in\puiseux{\FF}{t}^{d+1}$ and pick $\tau$ large enough to avoid their poles and zeros.
  Then, the scalar product of $x$ and $y$ vanishes if and only if the scalar product of $x(\tau)$ and $y(\tau)$ does.
  Hence, the claim follows.

  Now we can assume that the polyhedral cone $C$ is pointed, i.e., it does not contain any linear subspace of positive dimension. 
  If this is not the case the subsequent argument applies to the quotient $C/L$. 
  
  Employing orthogonal bases, as for the lineality spaces above, shows that the evaluation maps the linear hull of $C$ to the linear hull of $C(\tau)$, preserving the dimension.
  So we may assume that $C$ is full-dimensional, as otherwise the arguments below hold in the linear hull of $C$.

  Let $\ell\leq\tbinom{m}{d}$ be the number of $d$-element sets of linearly independent rows of the matrix $A$.
  For each such set of rows the set of solutions to the corresponding homogeneous system of linear equations is a one-dimensional subspace of $\puiseux{\FF}{t}^{(d+1)}$.
  For each such system of homogeneous linear equations pick two non-zero solutions, which are negatives of each other.
  We arrive at $2\ell$ vectors in $\puiseux{\FF}{t}^{(d+1)}$ which we use to form the columns of the matrix $Z \in \puiseux{\FF}{t}^{(d+1)\times 2\ell}$.

  By the Farkas--Minkowski--Weyl theorem, we may assume that the columns of $B$ from \eqref{eq:B} only consist of the rays of $C$ and that the rays of $C$ form a subset of the columns of $Z$.
  In particular, the columns of $B$ occur in $Z$.
  Since the cone $C$ is pointed, the matrix $B$ contains at most one vector from each opposite pair of the columns of $Z$.
  This entails that $B$ has at most $\ell$ columns.

  Further, the real matrix $Z(\tau)$ contains all rays of $C(\tau)$ for each $\tau$ that avoids the poles of $A$ and $Z$.
  In the following, we want to show that those columns of $Z(\tau)$ which form the rays of $C(\tau)$ are exactly the columns of $B(\tau)$.

  We define $s(j,k)\in\puiseux{\FF}{t}$ to be the scalar product of the $j$th row of $A$ and the $k$th column of $Z$.
  The $m\cdot 2\ell$ signs of the scalar products $s(j,k)$, for $j\in[m]$ and $k\in[2\ell]$, form the \emph{chirotope} of the linear hyperplane arrangement defined by the rows of $A$ (in fact, due to taking two solutions for each homogenous system of linear equations, we duplicate the information of the chirotope).
  For almost all $\tau \in \FF$ evaluating the Puiseux fractions $s(j,k)$ at $\tau$ yields an element of $\FF$.
  For sufficiently large $\tau$ the sign of $s(j,k)$ agrees with its evaluation.
  This follows from the definition of the ordering on $\puiseux{\FF}{t}$, cf. \cite[Proposition, \S 1.3]{Jeroslow:1973b}.

  Let $\tau_0\in\FF$ be larger than all the at most finitely many poles of $A$ and $Z$.
  Further, let $\tau_0$ be large enough such that the chirotope of $A(\tau)$ agrees with the chirotope of $A$ for all $\tau>\tau_0$.

  By construction the rays of $C$ correspond to the non-negative columns of the chirotope whose support, given by the $0$ entries,
  is maximal with respect to inclusion; these are exactly the columns of $B$.  The corresponding columns of the chirotope of $A(\tau)$, for $\tau>\tau_0$, yield
  the rays of $C(\tau)$, which, hence, are the columns of $B(\tau)$.

  The same holds for the facets of $C$ and $C(\tau)$.
  The facets of $C$ correspond to the non-negative rows of the chirotope whose support, given by the $0$ entries, is maximal with respect to inclusion.

  Now the claim follows since the face lattice of a polyhedral cone is determined by the incidences between the facets and the rays.
\qed
\end{proof}

A statement related to Theorem~\ref{thm:evaluation} occurs in Benchimol's PhD thesis~\cite{Benchimol:2014}.
The Proposition 5.12 in \cite{Benchimol:2014} discusses the combinatorial structure of tropical polyhedra (arising as the feasible regions of tropical linear programs).
Yet here we consider the relationship between the combinatorial structure of Puiseux polyhedra and their evaluations over the reals.
As in the proof of \cite[Proposition 5.12]{Benchimol:2014} we could derive an explicit upper bound on the optimal $\tau_0$.
To this end one can estimate the coefficients of the Puiseux fractions in $Z$, which are given by determinantal expressions arising from submatrices of $A$.
Their poles and zeros are bounded by Cauchy bounds (e.g., see \cite[Thm. 8.1.3]{RahmanSchmeisser:2002}) depending on those coefficients.
We leave the details to the reader.

\medskip

A \emph{convex polyhedron} is the intersection of finitely many linear inequalities.  It is a called a \emph{polytope}
if it is bounded.  Restricting to cones allows a simple description in terms of homogeneous linear inequalities.
Yet this encompasses arbitrary polytopes and polyhedra, as they can equivalently be studied through their homogenizations.
In fact, all implementations in \polymake are based on this principle.
For further reading we refer to \cite[\S1.5]{Ziegler:1995}.
We visualize Theorem~\ref{thm:evaluation} with a very simple example.
\begin{example} \label{ex:param_triang}
  Consider the polytope $P$ in $\puiseux{\RR}{t}^2$ for large $t$ defined by the four inequalities
  \begin{align*}
    x_1, x_2 \geq 0, \qquad x_1 + x_2 \leq 3 , \qquad x_1 - x_2 \leq t \enspace .
  \end{align*}
  The evaluations at $\tau\in\{0,1,3\}$ are depicted in Figure~\ref{fig:param_triang}.  For $\tau = 0$ we obtain a
  triangle, for $\tau = 1$ a quadrangle and for $\tau \geq 3$ a triangle again.  The latter is the combinatorial type of
  the polytope $P$ over the field of Puiseux fractions with real coefficients.
\end{example}

\begin{figure}[bt]
  \centering

\newcommand{\scalefactor}{0.6}

\newcommand{\xmin}{-0.5}
\newcommand{\xmax}{4.5}
\newcommand{\ymin}{-2.5}
\newcommand{\ymax}{4.5}
\tikzset{Hyp/.style = {
    color = blue!50!black, thick
}}

\tikzset{fillPoly/.style = {
    color = blue!30
}}

\newcommand{\puiseuxtriangle}[1]{
\fill[fillPoly] (0,0) -- (#1, 0) -- (1.5+0.5*#1, 1.5- 0.5*#1) -- (0, 3);
\draw[Hyp] (\xmin, -0.5 - #1) -- (\xmax, \xmax - #1);
\draw[Hyp] (\xmin, 0)-- (\xmax, 0);
\draw[Hyp] (0, \ymin) -- (0, \ymax);
\draw[Hyp] (\xmin, 3-\xmin) -- (\xmax, 3-\xmax);

}

\begin{tikzpicture}[scale=\scalefactor]
\Gitter{\xmin}{\xmax}{\ymin}{\ymax}
\Koordinatenkreuz{\xmin-0.3}{\xmax+0.4}{\ymin-0.1}{\ymax+0.4}{$x_1$}{$x_2$}
\puiseuxtriangle{0};
\end{tikzpicture}
\qquad
\begin{tikzpicture}[scale=\scalefactor]
\Gitter{\xmin}{\xmax}{\ymin}{\ymax}
\Koordinatenkreuz{\xmin-0.3}{\xmax+0.4}{\ymin-0.1}{\ymax+0.4}{$x_1$}{$x_2$}
\puiseuxtriangle{1};
\end{tikzpicture}
\qquad
\begin{tikzpicture}[scale=\scalefactor]
\Gitter{\xmin}{\xmax}{\ymin}{\ymax}
\Koordinatenkreuz{\xmin-0.3}{\xmax+0.4}{\ymin-0.1}{\ymax+0.4}{$x_1$}{$x_2$}
\fill[fillPoly] (0,0) -- (3, 0) -- (1.5+0.5*3, 1.5- 0.5*3) -- (0, 3);
\draw[Hyp] (\xmin+1, \ymin) -- (\xmax, \xmax - 3);
\draw[Hyp] (\xmin, 0)-- (\xmax, 0);
\draw[Hyp] (0, \ymin) -- (0, \ymax);
\draw[Hyp] (\xmin, 3-\xmin) -- (\xmax, 3-\xmax);

\end{tikzpicture}
  \caption{Polygon depending on a real parameter as defined in Example~\ref{ex:param_triang}}
  \label{fig:param_triang}
\end{figure}
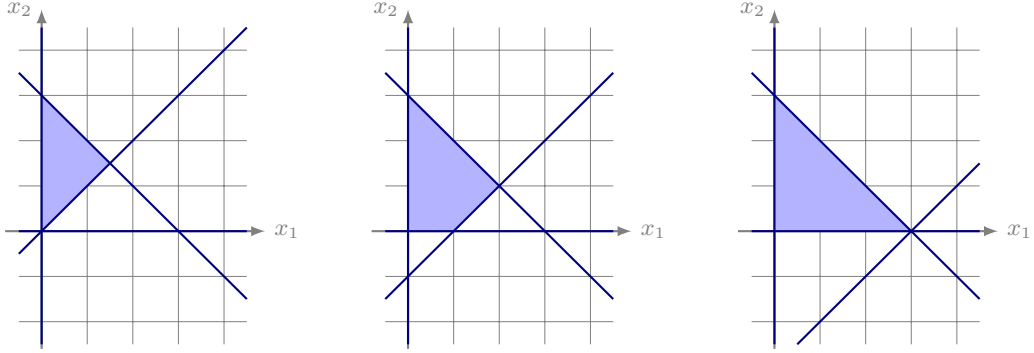

\begin{corollary}
  The set of combinatorial types of polyhedral cones which can be realized over $\puiseux{\FF}{t}$ is the same as over $\FF$.
\end{corollary}

\begin{proof}
  One inclusion is trivial since $\FF$ is a subfield of $\puiseux{\FF}{t}$.  The other inclusion follows from the
  preceding result.\qed
\end{proof}

For $A\in\puiseux{\FF}{t}^{m\times d}$, $b\in\puiseux{\FF}{t}^m$ and $c\in\puiseux{\FF}{t}^d$ we consider the linear
program $\LP(A,b,c)$ over $\puiseux{\FF}{t}$ which reads as
\begin{equation}\label{eq:LP}
  \begin{array}{r@{\quad}l}
    \text{maximize} & \transpose{c} \cdot x \\
    \text{subject to}
    &
    A \cdot x = b \,, \ x \geq 0 \enspace .
  \end{array}
\end{equation}
For each positive $\tau\in\FF$ (which avoids the poles of the Puiseux fractions which arise as coefficients) we obtain a
linear program $\LP(A(\tau),b(\tau),c(\tau))$ over $\FF$. Theorem~\ref{thm:evaluation} now has the following consequence
for parametric linear programming.

\begin{corollary}\label{cor:LP}
  Let $x^*\in\puiseux{\FF}{t}^d$ be an optimal solution to the LP~\eqref{eq:LP} with optimal value $v\in\puiseux{\FF}{t}$.
  Then there is a positive element $\tau_0\in\FF$ so that for every $\tau > \tau_0$ the vector $x^*(\tau)$ is an optimal
  solution for $\LP(A(\tau),b(\tau),c(\tau))$ with optimal value $v(\tau)$.
\end{corollary}

The above corollary was proved by Jeroslow~\cite[\S2.3]{Jeroslow:1973b}.  His argument, based on controlling signs of
determinants, is essentially a local version of our Theorem~\ref{thm:evaluation}.  Moreover, determining all the rays of
a polyhedral cone can be reduced to solving sufficiently many LPs.  This could also be exploited to derive another proof of
Theorem~\ref{thm:evaluation} from Corollary~\ref{cor:LP}.

\begin{remark}
  It is worth to mention the special case of a linear program over the field $\puiseux{\FF}{t}$, where the coordinates
  of the linear constraints, in fact, are elements of the field $\FF$ of coefficients, but the coordinates of the linear
  objective function are arbitrary elements in $\puiseux{\FF}{t}$.  That is, the feasible domain is a polyhedron, $P$,
  over~$\FF$.  Evaluating the objective function at some $\tau\in\FF$ makes one of the vertices of $P$ optimal.  Solving
  for all values of~$\tau$, in general, amounts to computing the entire normal fan of the polyhedron $P$.  This is
  equivalent to solving the dual convex hull problem over~$\FF$ for the given inequality description of $P$; see
  also~\cite{JCM:2008}.  Here we restrict our attention to solving parametric linear programs via Corollary~\ref{cor:LP}.
\end{remark}

The next example is a slight variation of a construction of Goldfarb and Sit~\cite{GoldfarbSit:1979}.  This is a class
of linear optimization problems on which certain versions of the simplex method perform poorly.

\begin{example} \label{ex:goldfarb_sit_cube}
  We fix $d > 1$ and pick a positive $\delta\leq \tfrac{1}{2}$ as well as a positive $\varepsilon < \tfrac{\delta}{2}$.
  Consider the linear program
  \[
  \begin{array}{r@{\quad}l}
    \text{maximize} & \, \sum_{i=1}^d \delta^{d-i} x_{i} \\[1ex]
    \text{subject to}
    &
    \begin{aligned}[t]
      0 &\leq x_1 \leq \varepsilon^{d-1} \\
      x_{j-1} &\leq \delta x_j \leq \varepsilon^{d-j} \delta- x_{j-1}  \qquad \text{ for } 2\leq j\leq d \enspace .
    \end{aligned}
  \end{array}
  \]
  The feasible region is combinatorially equivalent to the $d$-dimensional cube.  Applying the simplex method with the
  ``steepest edge'' pivoting strategy to this linear program with the origin as the start vertex visits all the $2^d$
  vertices.  Moreover, the vertex-edge graph with the orientation induced by the objective function is isomorphic to
  (the oriented vertex-edge graph of) the Klee--Minty cube~\cite{KleeMinty:1972}.  See Figure~\ref{fig:goldfarb_sit_cube}
  for a visualization of the $3$-dimensional case.

  We may interpret this linear program over the reals or over $\puiseux{(\puiseux{\RR}{\delta})}{\varepsilon}$, the
  field of Puiseux fractions in the indeterminate $\varepsilon$ with coefficients in the field $\puiseux{\RR}{\delta}$.
  This depends on whether we want to view $\delta$ and $\varepsilon$ as indeterminates or as real numbers.  Here we consider
  the ordering induced by the dual valuation $\val^*$, i.e., $\delta$ and $\varepsilon$ are \emph{small} infinitesimals,
  where $\varepsilon\ll\delta$.  Two more choices arise from considering $\varepsilon$ a constant in $\puiseux{\RR}{\delta}$
  or, conversely, $\delta$ a constant in $\puiseux{\RR}{\varepsilon}$.  Note that our constraints on $\delta$ and
  $\varepsilon$ are feasible in all four cases.
\end{example}


\begin{figure}[bt]
  \centering
\newcommand{\ptau}{0.75}
\newcommand{\pdelta}{2}
\tikzset{
    halfarrow/.style={postaction={decorate},
        decoration={markings,mark=at position .85 with
       {\arrow{stealth}}}}}
\begin{tikzpicture}[x  = {(1cm,0cm)},
                    y  = {(0cm,1cm)},
                    z  = {(0.3cm,0.6cm)},
                    scale = 3,
                    color = {black}]

  \coordinate (v_0) at (0,0,0);
  \coordinate (v_1) at (\ptau^2, \ptau^2/\pdelta , \ptau^2/\pdelta^2); 
  \coordinate (v_2) at (\ptau^2, \ptau-\ptau^2/\pdelta, \ptau/\pdelta-\ptau^2/\pdelta^2);
  \coordinate (v_3) at (0, \ptau,\ptau/\pdelta); 
  \coordinate (v_4) at (0, \ptau, 1-\ptau/\pdelta); 
  \coordinate (v_5) at (\ptau^2, \ptau-\ptau^2/\pdelta, 1-\ptau/\pdelta+\ptau^2/\pdelta^2); 
  \coordinate (v_6) at (\ptau^2, \ptau^2/\pdelta, 1-\ptau^2/\pdelta^2); 
  \coordinate (v_7) at (0,0,1);

 \draw (v_7) -- (v_0) -- (v_3);
 \draw (v_1) -- (v_6);
 \draw (v_2) -- (v_5);
 \draw (v_4) -- (v_7);

 \draw[ultra thick, halfarrow, blue!70!black] (v_6) -- (v_7);
 \draw[preaction={draw=white, line width=5pt}] (v_2) -- (v_5);
 \draw[ultra thick, blue!70!black] (v_5) -- (v_6);
 \draw[ultra thick, halfarrow, blue!70!black] (v_4) -- (v_5);
 \draw[ultra thick, blue!70!black] (v_3) -- (v_4);
 \draw[preaction={draw=white, line width=5pt},ultra thick, halfarrow, blue!70!black] (v_2) -- (v_3);
 \draw[ultra thick, blue!70!black] (v_1) -- (v_2);
 \draw[ultra thick, halfarrow, blue!70!black] (v_0) -- (v_1);

 \foreach \i in {0,1,...,7} {
  \fill[blue!70!black] (v_\i) circle (0.7pt);
 }

\end{tikzpicture}
  \caption{The $3$-dimensional Goldfarb--Sit cube.}
  \label{fig:goldfarb_sit_cube}
\end{figure}
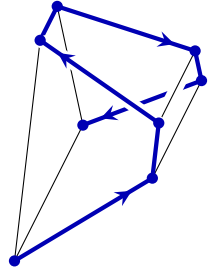

\newlength{\mytemplen}
\newcommand{\mybackup}[1]
{
  \settowidth{\mytemplen}{\(\displaystyle #1\)}
  \hskip-\mytemplen%
  \mkern-8mu
}

Our third and last example is a class of linear programs occurring in~\cite{ABGJ:1405.4161}.  For these the central path
of the interior point method with a logarithmic barrier function has a total curvature which is exponential as a function
of the dimension.

\begin{example} \label{ex:long_polytope}
  Given a positive integer $r$, we define a linear program over the field $\puiseux{\QQ}{t}$ (with the primal valuation) in the $2r+2$ variables
  $u_0, v_0, u_1, v_1, \dots, u_r, v_r$ as follows:
  \[
  \begin{array}{r@{\quad}l}
    \text{minimize} & \,  v_0 \\[1ex]
    \text{subject to} 
    &
    \begin{aligned}[t]
      u_0 & \leq t \, , \  v_0 \leq t^2 \\
      & \left.\mybackup{ u_i}\begin{aligned}
          u_i & \leq t  u_{i-1} \, , \  u_i \leq t v_{i-1} \vphantom{t^{1 - \frac{1}{2}}} \\
          v_i & \leq t^{1 - \frac{1}{2^{i}}} ( u_{i-1} +  v_{i-1}) 
        \end{aligned}\; \right\} \qquad \text{for} \ 1\leq i\leq r\\
      u_r & \geq 0 \, , \   v_r \geq 0 \, . \vphantom{t^{1 - \frac{1}{2}}}  
    \end{aligned}
  \end{array}
  \]

  Here it would be interesting to know the exact value for the optimal $\tau_0$ in Theorem~\ref{thm:evaluation}, as a
  function of $r$.  Experimentally, based on the method described below, we found $\tau_0=1$ for $r=1$ and $\tau_0=2^{2^{r-1}}$ for $r$ at most $5$.
  We conjecture the latter to be the true bound in general.
\end{example}

To find the optimal bound for a given constraint matrix $A$ we can use the following method.
One can solve the dual convex hull problem for the cone $C$, which is the feasible region in homogenized form, to obtain a matrix $B$ whose columns are the rays of $C$.
This also yields a submatrix of $A$ corresponding to the rows which define facets of $C$.
Without loss of generality we may assume that that submatrix is $A$ itself.
Let $\tau_0$ be the largest zero or pole of any (Puiseux fraction) entry of the matrix $A \cdot B$.
Then for every value $\tau > \tau_0$ the sign patterns of $(A \cdot B)(\tau)$ and $A \cdot B$ coincide, and so do the combinatorial types of $C$ and $C(\tau)$.
Determining the zeros and poles of a Puiseux fraction amounts to factorizing univariate polynomials.


\section{Tropical Dual Convex Hulls}\label{sec:tropical}
\noindent
Tropical geometry is the study of the piecewise linear images of algebraic varieties, defined over a field with a
non-Archimedean valuation, under the valuation map; see~\cite{Tropical+Book} for an overview.  The motivation for
research in this area comes from at least two different directions.  First, tropical varieties still retain a lot of
interesting information about their classical counterparts.  Therefore, passing to the tropical limit opens up a path
for combinatorial algorithms to be applied to topics in algebraic geometry.  Second, the algebraic geometry perspective
offers opportunities for optimization and computational geometry.  Here we will discuss how classical convex hull
algorithms over fields of Puiseux fractions can be applied to compute tropical convex hulls; see
\cite{Tropical+Convex+Hull+Computations} for a survey on the subject;
a standard algorithm is the tropical double description method of \cite{AllamigeonGaubertGoubault:2010}.

The \emph{tropical semiring} $\TT$ consists of the set $\RR\cup\{-\infty\}$ together with $u \oplus v = \max(u,v)$ as
the addition and $u \odot v = u+v$ as the multiplication.  Extending these operations to vectors turns $\TT^{d+1}$ into
a semimodule.  A \emph{tropical cone} is the sub-semimodule 
\[\tcone(G) = \SetOf{\lambda_1 \odot g_1 \oplus \cdots \oplus \lambda_n \odot g_{n}}{\lambda_1,\ldots,\lambda_{n} \in \TT}\]
generated from the columns $g_1,\dots,g_n$ of the matrix
$G \in \TT^{(d+1)\times n}$.  Similar to classical cones, tropical cones admit an exterior
description~\cite{GaubertKatz:2011}.  It is known that every tropical cone is the image of a classical cone under the
valuation map $\val \colon \pseries{\RR}{t} \rightarrow \TT$; see \cite{DevelinYu:2007}.  Based on this idea, we present
an algorithm for computing generators of a tropical cone from a description in terms of tropical linear inequalities;
see Algorithm~\ref{algo:halfToGen} below.

Before we can start to describe that algorithm we first need to discuss matters of general position in the tropical
setting.  The \emph{tropical determinant} of a square matrix $U \in \TT^{\ell\times \ell}$ is given by
\begin{equation}\label{eq:tdet}
  \tdet(U) \ = \ \bigoplus_{\sigma \in S_{\ell}} u_{1\pi(1)} \odot \cdots \odot u_{\ell\pi(\ell)} \enspace.
\end{equation}
Here $S_\ell$ is the symmetric group of degree $\ell$; computing the tropical determinant is the same as solving a
linear assignment optimization problem. Consider a pair of matrices $H^+, H^- \in \TT^{m\times(d+1)}$ which
serve as an exterior description of the tropical cone
\begin{equation}\label{eq:Q}
  Q \ = \ \SetOf{z \in \TT^{(d+1)}}{H^+\odot z \geq H^- \odot z} \enspace .
\end{equation}
In contrast to the classical situation we have to take two matrices into account.  This is due to the lack of an
additive inverse operation.  We will assume that $\mu(i,j):=\min(H^+_{ij}, H^-_{ij}) = -\infty$ for any pair $(i,j) \in
[m]\times[d+1]$, i.e., for each coordinate position at most one of the corresponding entries in the two matrices is
finite.  Then we can define
\[
\chi(i,j) \ := \begin{cases} 1 & \text{if $\mu(i,j)=H^+_{ij}\neq-\infty$}\\ -1 & \text{if $\mu(i,j)=H^-_{ij}\neq-\infty$}
  \\ 0 & \text{otherwise} \enspace . \end{cases}
\]
For each term $u_{1\pi(1)} \odot \cdots \odot u_{\ell\pi(\ell)}$ in \eqref{eq:tdet} we define its \emph{sign} as
\[
\sgn(\pi) \cdot \chi(1,\pi(1)) \cdots \chi(\ell,\pi(\ell)) \enspace ,
\]
where $\sgn(\pi)$ is the sign of the permutation $\pi$.  Now the exterior description \eqref{eq:Q} of the tropical cone
$Q$ is \emph{tropically sign-generic} if for each square submatrix $U$ of $H^+ \oplus H^-$ we have $\tdet(U)\neq-\infty$
and, moreover, the signs of all terms $u_{1\pi(1)} \odot \cdots \odot u_{\ell\pi(\ell)}$ which attain the maximum in
\eqref{eq:tdet} agree.  By looking at $1{\times}1$-submatrices $U$ we see that in this case all coefficients of the matrix
$H^+\oplus H^-$ are finite and thus $\chi(i,j)$ is never $0$.

\medskip

\begin{algorithm}[H]
 \caption{A dual tropical convex hull algorithm}
 \label{algo:halfToGen}
  \dontprintsemicolon

  \KwIn{pair of matrices $H^+, H^- \in \TT^{m\times(d+1)}$ which provide a tropically sign-generic exterior description of
    the tropical cone $Q$ from \eqref{eq:Q}}
  \KwOut{generators for $Q$}

  pick two matrices $A^+, A^- \in \pseries{\RR}{t}^{m\times (d+1)}$ with strictly positive entries such that $\val(A^+)
  = H^+$ and $\val(A^-) = H^-$ \;

  apply a classical dual convex hull algorithm to determine a matrix $B \in \pseries{\RR}{t}^{(d+1)\times n}$ such that
  $\SetOf{B\cdot a}{a \in \pseries{\RR}{t}^{n}, \, a \geq 0} = \SetOf{x \in \pseries{\RR}{t}^{(d+1)}}{(A^+-A^-)\cdot x\geq0, \, x \geq 0}$ \;

  \Return $\val(B)$ \;

\end{algorithm}

\begin{proof}[Correctness of  Algorithm~\ref{algo:halfToGen}]
  The main lemma of tropical linear programming \cite[Theorem~16]{AlBeGaJo:2015} says the following.  In the tropically
  sign-generic case, an exterior description of a tropical cone can be obtained from an exterior description of a
  classical cone over Puiseux series by applying the valuation map to the constraint matrix coefficient-wise.  This
  statement assumes that the classical cone is contained in the non-negative orthant.  We infer that
  \[
  \begin{aligned}
    Q \ &= \ \SetOf{z \in \TT^{m\times(d+1)}}{H^+ \odot z \geq H^- \odot z}\\ 
    &= \ \val\left(\SetOf{x \in \pseries{\RR}{t}^{m\times(d+1)} }{A^+\cdot x \geq A^- \cdot x, x \geq 0}\right) \\
    &= \ \val\left(\SetOf{\vphantom{\bigl(} B\cdot a}{a \in \pseries{\RR}{t}^{n}, x \geq 0}\right) \enspace .
  \end{aligned}
  \]
  Now \cite[Proposition 2.1]{DevelinYu:2007} yields $Q = \val(\SetOf{B\cdot a}{a \in \pseries{\RR}{t}^{n}, x \geq 0}) =
  \tcone(\val(B))$.  This ends the proof.
  \qed
\end{proof}

The correctness of our algorithm is not guaranteed if the genericity condition is not satisfied.
The crucial properties of the lifted matrices $A^+, A^-$ are not necessarily fulfilled.
It is an open question of how an exterior description over $\TT$ is related to an exterior description over $\pseries{\RR}{t}$ in the general setting.
We are even lacking a convincing concept for the ``facets'' of a general tropical cone.

\section{Implementation}\label{sec:implementation}
\noindent
As a key feature the \polymake system is designed as a \perl/\cplusplus hybrid, that is, both programming languages are
used in the implementation and also both programming languages can be employed by the user to write further code.  One
main advantage of Perl is the fact that it is interpreted; this makes it suitable as the main front end for the user.
Further, Perl has its strengths in the manipulation of strings and file processing.  \cplusplus on the other hand is a
compiled language with a powerful template mechanism which allows to write very abstract code which, nonetheless, is
executed very fast.  Our implementation, in \cplusplus, makes extensive use of these features.  The implementation of
the dual steepest edge simplex method, contributed by Thomas Opfer, and the beneath-beyond method for computing convex
hulls (see \cite{Edelsbrunner:1987} and \cite{Joswig:2003}) are templated.  Therefore \polymake can handle both
computations for arbitrary number field types which encode elements in an ordered field.

Based on this mechanism we implemented the type \texttt{RationalFunction} which depends on two generic template types for
coefficients and exponents.  Note that the field of coefficients here does not have to be ordered.
Our proof-of-concept implementation employs the classical Euclidean GCD
algorithm for normalization.  Currently the numerator and the denominator are chosen coprime such that the denominator
is normalized with leading coefficient one.  For the most interesting case $\FF=\QQ$ it is known that the coefficients
of the intermediate polynomials can grow quite badly, e.g., see \cite[Example~1]{vonzurGathenGerhard:2003}.  Therefore,
as expected, this is the bottleneck of our implementation.  In a number field or in a field with a non-Archimedean valuation 
the most natural choice for a normalization is to pick the elements of the ring of integers as coefficients.  The reason 
for our choice is that this more generic design does not make any assumption on the field of coefficients.
This makes it very versatile, and it fits the overall programming style in \polymake.  A fast specialization to the rational
coefficient case could be based on \cite[Algorithm~11.4]{vonzurGathenGerhard:2003}.  This is left for a future version.

The \polymake implementation of Puiseux fractions $\puiseux{\FF}{t}$ closely follows the construction described in
Section~\ref{sec:puiseux}. The new number type is derived from \texttt{RationalFunction} with overloaded comparison
operators and new features such as evaluating and converting into \texttt{TropicalNumber}. An extra template parameter
\texttt{MinMax} allows to choose whether the indeterminate $t$ is a small or a large infinitesimal.

There are other implementations of Puiseux series arithmetic, e.g., in \texttt{Magma} \cite{Magma} or
\texttt{MATLAB}~\cite{Matlab}.  However, they seem to work with finite truncations of Puiseux series and floating-point
coefficients.  This does not allow for exact computations of the kind we are interested in.

\section{Computations}\label{sec:computations}
\noindent
We briefly show how our \polymake implementation can be used.  Further, we report on timings for our LP solver, tested
on the Goldfarb--Sit cubes from Example~\ref{ex:goldfarb_sit_cube}, and for our (dual) convex hull code, tested on the
polytopes with a ``long and winded'' central path from Example~\ref{ex:long_polytope}.

\subsection{Using \polymake}
The following code defines a $3$-dimensional Goldfarb--Sit cube over the field $\puiseux{\QQ}{t}$, see
Example~\ref{ex:goldfarb_sit_cube}.  We use the parameters $\varepsilon = t$ and $\delta = \tfrac{1}{2}$.  The template
parameter \texttt{Min} indicates that the ordering is induced by the dual valuation $\val^*$, and hence the
indeterminate $t$ plays the role of a small infinitesimal.
\begin{lstlisting}
polytope > $monomial=new UniMonomial<Rational,Rational>(1);
polytope > $t=new PuiseuxFraction<Min>($monomial);
polytope > $p=goldfarb_sit(3,2*$t,1/2); 
\end{lstlisting}
The polytope object, stored in the variable \verb+$p+, is generated with a facet description from which further
properties will be derived below.  It is already equipped with a \texttt{LinearProgram} subobject encoding the
objective function from Example~\ref{ex:goldfarb_sit_cube}. The following lines show the maximal value and corresponding
vertex of this linear program as well as the vertices derived from the outer description. Below, we present timings for
such calculations.
\enlargethispage{1em}
\begin{lstlisting}
polytope > print $p->LP->MAXIMAL_VALUE;
(1)
polytope > print $p->LP->MAXIMAL_VERTEX;
(1) (0) (0) (1)
polytope > print $p->VERTICES;
(1) (0) (0) (0)
(1) (t^2) (2*t^2) (4*t^2)
(1) (0) (t) (2*t)
(1) (t^2) (t -2*t^2) (2*t -4*t^2)
(1) (0) (0) (1)
(1) (t^2) (2*t^2) (1 -4*t^2)
(1) (0) (t) (1 -2*t)
(1) (t^2) (t -2*t^2) (1 -2*t + 4*t^2)
\end{lstlisting}
As an additional benefit of our implementation we get numerous other properties for free.  For instance, we can compute
the parameterized volume, which is a polynomial in $t$.
\begin{lstlisting}
polytope > print $p->VOLUME;
(t^3 -4*t^4 + 4*t^5)
\end{lstlisting}
That polynomial, as an element of the field of Puiseux fractions, has a valuation, and we can evaluate it at the
rational number $\tfrac{1}{12}$.
\begin{lstlisting}
polytope > print $p->VOLUME->val;
3
polytope > print $p->VOLUME->evaluate(1/12);
25/62208
\end{lstlisting}

\subsection{Linear programs}
We have tested our implementation by computing the linear program of Example~\ref{ex:goldfarb_sit_cube} with polyhedra
defined over Puiseux fractions.

The simplex method in \polymake is an implementation of a (dual) simplex with a (dual) steepest edge pricing.  We set up
the experiment to make sure our Goldfarb--Sit cube LPs behave as badly as possible.  That is, we force our
implementation to visit all $n=2^d$ vertices, when $d$ is the dimension of the input.  Table~\ref{tab:goldfarb_sit}
illustrates the expected exponential growth of the execution time of the linear program.  In three of our four
experiments we choose $\delta$ as $\tfrac{1}{2}$.  The computation over $\puiseux{\QQ}{\varepsilon}$ costs a factor of
about $80$ in time, compared with the rational cubes for a modest $\varepsilon = \tfrac{1}{6}$.  However, taking a small
$\varepsilon$ whose binary encoding takes more than $18,\! 000$ bits is substantially more expensive than the computations
over the field $\puiseux{\QQ}{\varepsilon}$ of Puiseux fractions.  Taking $\delta$ as a second small infinitesimal is
possible but prohibitively expensive for dimensions larger than twelve.

\begin{table}[th]
\centering
\caption{Timings (in seconds) for the Goldfarb--Sit cubes of dimension~$d$ with $\delta=\tfrac{1}{2}$.  For $\varepsilon$
  we tried a small infinitesimal as well as two rational numbers, one with a short binary encoding and another one whose
  encoding is fairly large.  For comparison we also tried both parameters as indeterminates.} 
\label{tab:goldfarb_sit}
\renewcommand{\arraystretch}{0.9}
\begin{tabular*}{0.75\linewidth}{@{\extracolsep{\fill}}rrr@{\hspace{1cm}}d{3}d{3}d{3}d{3}@{}}
      \toprule
   $d$ & $m$ & $n$  & \multicolumn{1}{c}{$\puiseux{\QQ}{\varepsilon}$} & \multicolumn{1}{c}{$\QQ$} & \multicolumn{1}{c}{$\QQ$} & \multicolumn{1}{c}{ $(\puiseux{\QQ}{\delta})\{\varepsilon\}$ } \\
       &     &      & \multicolumn{1}{c}{$\varepsilon$}  & \multicolumn{1}{c}{$\varepsilon=\frac{1}{6}$} & \multicolumn{1}{c}{$\varepsilon=\tfrac{2}{17^{4500}}$} & \multicolumn{1}{c}{$\varepsilon\ll\delta$} \\
      \midrule
    3 &  6 &    8 &  0.010 & 0.003 &   0.005 &    0.101  \\
    4 &  8 &   16 &  0.026 & 0.001 &   0.017 &    0.353  \\
    5 & 10 &   32 &  0.064 & 0.002 &   0.065 &    1.034  \\
    6 & 12 &   64 &  0.157 & 0.007 &   0.253 &    2.877  \\
    7 & 14 &  128 &  0.368 & 0.006 &   0.829 &    7.588  \\
    8 & 16 &  256 &  0.843 & 0.016 &   2.643 &   19.226  \\
    9 & 18 &  512 &  1.906 & 0.039 &   7.703 &   47.806  \\
   10 & 20 & 1024 &  4.258 & 0.090 &  21.908 &  118.106  \\
   11 & 22 & 2048 &  9.383 & 0.191 &  59.981 &  287.249  \\
   12 & 24 & 4096 & 20.583 & 0.418 & 160.894 &  687.052  \\
    \bottomrule
  \end{tabular*}
\end{table}

\subsection{Convex hulls}
We have also tested our implementation by computing the vertices of the polytope from
Example~\ref{ex:long_polytope}. For this we used the client \texttt{long\_and\_winding} which creates the $d =
(2r+2)$-dimensional polytope given by $m= 3r+4$ facet-defining inequalities.  Over the rationals we evaluated the
inequalities at $2^{2^r}$ which probably gives the correct combinatorics; see the discussion at the end of
Example~\ref{ex:long_polytope}.  This very choice forces the coordinates of the defining inequalities to be integral, such
that the polytope is rational.  The number of vertices $n$ is derived from that rational polytope.
The running times grow quite dramatically for the parametric input.
This overhead could be reduced via a better implementation of the Puiseux fraction arithmetic.
\begin{table}[th]
\centering
\caption{Timings (in seconds) for convex hull computation of the feasibility set from Example~\ref{ex:long_polytope}. All timings represent an average over ten iterations. If any test exceeded a one hour time limit this and all larger instances of the experiment were skipped and marked $-$.}
\label{tab:ch}
\renewcommand{\arraystretch}{0.9}
\begin{tabular*}{0.5\linewidth}{@{\extracolsep{\fill}}r@{\hspace{1cm}}rrr@{\hspace{1cm}}d{3}d{3}@{}}
      \toprule
    $r$ & $d$ & $m$ & $n$  & \multicolumn{1}{c}{$\puiseux{\QQ}{t}$} & \multicolumn{1}{c}{$\QQ$} \\
      \midrule
    1 &  4 &   7  &     11 &     0.018  &    0.000 \\
    2 &  6 &  10  &     28 &     0.111  &    0.000 \\
    3 &  8 &  13  &     71 &     0.754  &    0.010 \\
    4 & 10 &  16  &    182 &    15.445  &    0.036 \\
    5 & 12 &  19  &    471 &  1603.051  &    0.150 \\
    6 & 14 &  22  &   1226 &     -      &    0.737 \\
    7 & 16 &  25  &   3201 &     -      &    4.001 \\
    8 & 18 &  28  &   8370 &     -      &   25.093 \\
    9 & 20 &  31  &  21901 &     -      &  223.240 \\
   10 & 22 &  34  &  57324 &     -      & 1891.133 \\
    \bottomrule
  \end{tabular*}
\end{table}

\goodbreak

\subsection{Experimental setup}
Everything was calculated on the same Linux machine with \polymake perpetual beta version 2.15-beta3 which includes the new number type, the templated simplex algorithm and the templated beneath-and-beyond convex hull algorithm.
All timings were measured in CPU seconds and averaged over ten iterations. The simplex algorithm was set to use only one thread.

All tests were done on \texttt{openSUSE}~13.1 (x86\_{}64), with Linux kernel~3.11.10-25, \texttt{clang}~3.3 and \texttt{perl}~5.18.1.
The rational numbers use a \cplusplus-wrapper around the \texttt{GMP} library version 5.1.2. As memory allocator \polymake uses the \texttt{pool\_allocator} from \texttt{libstdc++}, which was version 4.8.1 for the experiments.

The hardware for all tests was:
\begin{quote}
  Intel(R) Core(TM) i7-3930K CPU @ 3.20GHz\\
  bogomips: 6400.21\\
  MemTotal: 32928276 kB
\end{quote}

\bibliographystyle{splncs03}
\bibliography{LPCHRF}

\end{document}